\documentclass[11pt,a4paper]{article}

\usepackage[OT1]{fontenc}
\usepackage[latin1]{inputenc}
\usepackage[english]{babel}

\usepackage{amsfonts}
\usepackage{amsthm}
\usepackage{amsmath}
\usepackage{amsfonts}
\usepackage{latexsym}
\usepackage{amssymb}

\newtheorem{teo}{Theorem}[section]
\newtheorem{pro}[teo]{Proposition}
\newtheorem{lem}[teo]{Lemma}
\newtheorem{cor}[teo]{Corollary}
\newtheorem{fed}[teo]{Definition}
\theoremstyle{definition}
\newtheorem{rem}[teo]{Remark}

\newtheorem{problem}{Problem}

\def\h{{\cal H}}

\def\nf{ {\|\cdot\|_{\phi}}}

\def\h{{\mathcal H}}

\def\bh{{\mathcal B}({\mathcal H})}

\def\kh{ {\mathcal K}(\h) }
\def\u{ {\mathcal U}({\mathcal H}) }
\def\ufi{ {\mathcal U}_{\phi} }
\def\nf{ {\|\cdot\|_{\phi}}}

\def\bdem{\begin{proof}}
\def\edem{\end{proof}}

\def\eps{\varepsilon}

\def\la{\lambda}

\def\La{\Lambda}

\def\N{\mathbb{N}}   
\def\R{\mathbb{R}}   
\def\C{\mathbb{C}}   

\def\cG{\mathcal{G}}

\def\G{\mathcal{G}}

\def\ele{\mathcal{L}}

\def\ese{\mathcal{S}}

\def\ete{\mathcal{Y}}
\def\eme{\mathcal{X}}
\def\eze{\mathcal{Z}}

 \DeclareMathOperator{\tr}{tr}

\newcommand{\peso}[1]{ \quad \text{ \rm  #1 } \quad }

\newcommand{\op}{B(\mathcal{H})}

\newcommand{\mat}{\mathcal{M}_n(\mathbb{C})}

\newcommand{\matsa}{\mathcal{H}(n)}
\newcommand{\matu}{\mathcal{U}(n)}

\newcommand{\matinv}{\mathcal{G}\textit{l}\,(n)}

\newcommand{\avi}[2]{\la_{#1}\left( #2\right)}
\newcommand{\svi}[2]{s_{#1}\left( #2\right)}

\newcommand{\sub}[2]{{#1}_{\mbox{\tiny{${#2}$}}}}

\def\trin{\hskip 3pt \hbox{\rm $\|$\hskip -6.9pt $\mid$ \hskip 3pt}}
\newcommand{\nnui}[1]{\trin #1 \trin}

\begin{document}


\title{Optimal paths for symmetric actions in the unitary group}
%
\date{}
\author{Jorge Antezana, Gabriel Larotonda and Alejandro Varela}

\maketitle

\begin{abstract}
Given a positive and unitarily invariant Lagrangian $\mathcal{L}$ defined in the algebra of Hermitian matrices, and a fixed interval $[a,b]\subset\mathbb R$, we study the action defined in the Lie group of $n\times n$ unitary matrices $\mathcal{U}(n)$ by
$$
\ese(\alpha)=\int_a^b \mathcal{L}(\dot\alpha(t))\,dt\,,
$$
where $\alpha:[a,b]\to\mathcal{U}(n)$ is a rectifiable curve. We prove that the one-parameter subgroups of $\mathcal{U}(n)$ are the optimal paths, provided the spectrum of the exponent is bounded by $\pi$.  Moreover, if $\mathcal{L}$ is strictly convex, we prove that one-parameter subgroups are the unique optimal curves joining given endpoints. Finally, we also study the connection of these results with unitarily invariant metrics in $\mathcal{U}(n)$ as well as angular metrics in the Grassmann manifold. 
\end{abstract}


\section{Introduction}

The group of $n\times n$ complex unitary matrices $\matu$ carries, as any Lie group, a canonical connection without torsion defined on left-invariant vector fields $X,Y$ as $\nabla_XY=\frac12 [X,Y]$, whose geodesics are the one-parameter groups  $t\mapsto Ue^{tZ}$ (here $U$ is a unitary matrix and $Z$ an anti-Hermitian matrix). We can introduce a Riemannian metric on the unitary group in a standard fashion
$$
\langle X,Y\rangle_g=Tr(U^*X(U^*Y)^*)=Tr(XY^*),
$$
for $U^*X,U^*Y$ in the Lie algebra of the group, that is, for $U^*X,U^*Y$ anti-Hermitian matrices. It is well-known that the connection just introduced is in fact the Levi-Civita connection of the metric $g$ induced by the trace, and that geodesics are short provided the spectrum of $Z$ is bounded by $\pi$ (see for instance \cite{andcan}).

Now consider the bi-invariant Finsler metric given by the spectral norm,
$$
\sub{\|X\|}{U}=\|U^*X\|=\|X\|
$$
for any $X$ tangent to a unitary matrix $U$. Remarkably, if one keeps the connection but changes the metric, the geodesics of the connection are still short for the induced rectifiable distance (which, as in the Riemannian setting, is computed as the infimum of the length of piecewise smooth curves joining given endpoints, and $L(\alpha)=\int_0^1 \|\dot{\alpha}\|dt$).  The same result was also proved in \cite{upe}, using techniques of variational calculus, if the Finsler metrics are given by the $p$-Schatten norms for $p\geq 2$. This raises a natural question: what do these norms have in common that could imply this phenomenon? A possible answer could be that all these norms are unitarily invariant, thus they induce bi-invariant metrics on the unitary group.  One of the main obstacles to deal with general unitarily invariant norms, is that variational arguments become untractable if the norm is not smooth enough.

In this article we prove that this is the right answer, and introduce a new approach that simplifies considerably the technicalities. It is based in a beautiful and deep result due to Thompson on the product of exponential matrices (Theorem \ref{porron} below).

Our approach also works for more general optimization problems described as follows: fix a bounded interval $[a,b]\subset\mathbb R$, and let $\ese$ be the action defined on piecewise $C^1$ curves  $\alpha:[a,b]\to\matu$ by
$$
\ese(\alpha)=\int_a^b \mathcal{L}(\dot\alpha(t))\,dt,
$$
where $\ele$ is a Lagrangian defined in the algebra of  $n\times n$ matrices, with the following unitary invariance property: for every $n\times n$ matrix $A$,  and every pair of $n\times n$ unitary matrices $U$ and $V$
\begin{equation}\label{L nui}
\ele(UAV)=\ele(A).
\end{equation}
As usual, it is asked that the Lagrangian is a convex and positive map, and without loss of generality we will assume that $\mathcal L(0)=0$. A Lagrangian that satisfies these properties will be called \textit{symmetric Lagrangian}. Two classical examples of symmetric Lagrangians are:
\begin{itemize}
\item An unitarily invariant norm $\|\cdot\|_\phi$;
\item The kinetic energy $E(A)=\|A\|_F^2$, where $\|\cdot\|_F$ denotes the Frobenius norm.
\end{itemize}
In the first case, we recover the geometric context mentioned above, because  the action $\ese$ defines the length of $\alpha$ associated to the Finsler structure that considers the norm $\|\cdot\|_\phi$ in each tangent space. Note that in this case, $\ese$ does not depend on the parametrization of $\alpha$. So, there is no significative difference between  the  problem of finding a curve that minimizes $\ese$ among all piecewise $C^1$ curves or among all piecewise $C^1$ curves with a given interval of parameters.

However, in the second example, the action associated to the kinetic energy depends on the parametrization. Let $\alpha:[a,b]\to\matu$ be a smooth curve. A simple change of variable shows that, if we take the family of curves $\alpha_r:[ra,rb]\to\matu$ defined by $\alpha_r(t)=\alpha(t/r)$, then $r\mapsto \ese(\alpha_r)$ is a non-increasing function for $r\in(0,+\infty)$. The same phenomenon also holds for any other convex Lagrangian. This suggests that in order to find a minimum we should fix the length of the interval of parameters. This is also suggested by considering the example of the energy functional, where the parameter $t$ should be interpreted as the time parameter.

As translations of that interval do not change the value of $\ese(\alpha)$, without lost of generality we can consider intervals of the form $[0,b]$. So, the optimization problem that we will study is the following:

\begin{problem}\label{problema}
Given $U,V\in\matu$ and $b>0$, find the piecewise $C^1$ curves $\gamma:[0,b]\to\matu$ such that $\gamma(0)=U$, $\gamma(b)=V$ and $\gamma$ minimizes the action given by
\begin{equation}\label{accion}
\ese(\alpha)=\int_0^b \mathcal{L}(\dot\alpha(t))\,dt
\end{equation}
where $\ele$ is a given symmetric Lagrangian.
\end{problem}

The second question that arises is whether the minimal paths, when they exist, are unique or not, or if they are unique modulus a reparametrization of the path. Thus we will study the following:

\begin{problem}\label{problema2}
Given $U,V\in\matu$, $b>0$, and a minimizing function $\gamma:[0,b]\to\matu$ with $\gamma(0)=U$, $\gamma(b)=V$, is this function the unique minimizer of the Lagrangian for the given endpoints? Is it true that any other minimizing curve with this given endpoints is just a reparametrization of $\gamma$?
\end{problem}

\section{Preliminaries}\label{sn}

Throughout this paper $\mat$ denotes the algebra of complex $n\times n$ matrices, $\matinv$ the group of all invertible elements of $\mat$,  $\matu$ the group of unitary $n\times n$ matrices, and $\matsa$ the real subalgebra of Hermitian matrices.  If $T\in \mat$, then $\|T\|$ stands for the usual spectral norm,  $|\cdot|$  indicates the modulus of $T$, i.e. $|T|=\sqrt{T^*T}$, and $\tr(T)$ denotes the trace of $T$.
Given $A\in\matsa$, $\avi{1}{A}\geq \ldots\geq \avi{n}{A}$ denotes the eigenvalues of $A$ arranged in non-increasing way, and given an arbitrary matrix $T\in\mat$, $\svi{1}{T}\geq \ldots\geq \svi{n}{T}$ denotes the singular values of $T$, i.e. the eigenvalues of $|T|$. We will use $\lambda(A)$ (resp. $s(T)$) to denote the vector in $\mathbb R^n$ consisting of the eigenvalues of $A$ (resp. the singular values of $T$). Finally, given $A,B\in\matsa$, by means of $A\leq B$ we denote that $A$ is less that or equal to $B$ with respect to the L\"owner order.

\subsection{Product of exponentials}

We begin this subsection with the following remarkable result:

\begin{teo}[Thompson \cite{thompson}]\label{porron}
Given $X,Y\in \matsa$, there exist unitary matrices $U$ and $V$ such that
$$
e^{i X}e^{i Y}=e^{i (UXU^*+VYV^*)}\,.
$$
\end{teo}

We will use the following corollary of Thompson's theorem:

\begin{cor}\label{corti}
Let $X,Y,Z\in \matsa$ be such that $\|Z\|\leq \pi$ and $e^{iX}e^{iY}=e^{iZ}$. Then, there are unitary matrices $U$ and $V$ such that
$|Z|\leq|UXU^*+VYV^*|$.
\end{cor}
\begin{proof}
By Thompson's Theorem it is enough to prove that,  if $X,Y\in \matsa$,  $e^{iX}=e^{iY}$, and  $\|X\|\leq\pi$, then $|X|\le |Y|$. Let $Y=\sum_{n\in\N} \eta_n\ e_n\otimes e_n$ be a spectral decomposition of $Y$. If $\La=\{n: e^{i\eta_n}=-1\}$, then
$$
|X|=\pi P+\sum_{n\notin\La} |\mu_n|\, e_n\otimes e_n\,,
$$
where $P$ is the spectral projection of $X$ onto the subspace generated by the eigenvectors associated to $\pm\pi$, and  the eigenvalues $\mu_n\in(-\pi,\pi)$ satisfy that $e^{i\mu_n}=e^{i\eta_n}$ for every $n\notin\La$. Clearly $PY=YP$ and $P|X|P\leq P|Y|P$. On the other hand, since $|\mu_n|\leq |\eta_n|$ for every $n\notin\La$, we also obtain that $(1-P)|X|(1-P)\leq (1-P)|Y|(1-P)$.
\end{proof}

Another result due to Thompson is the following triangle inequality for the modulus of matrices:

\begin{teo}[Thompson \cite{[T1],[T2]}]\label{modulos}
Given $A,B\in\mat$, there exist unitaries $V$ and $W$ such that
$$
|X+Y|\leq V|X|V^*+W|Y|W^*.
$$
\end{teo}

Combining this result with Corollary \ref{corti} we get:
\begin{pro}\label{muchas}
Let $m\geq 2$, and consider $X,X_1,\ldots,X_m\in\matsa$ such that $\|X\|\leq \pi$ and
$$
e^{iX}=e^{iX_1}\cdots e^{iX_m}\,.
$$
Then, there exist unitary matrices $U_1,\ldots, U_m$ such that
$\displaystyle |X|\leq \sum_{k=1}^{m} U_k|X_k|U_k^*$.
\end{pro}
\begin{proof}
For $m=2$ it is a direct consequence of Corollary \ref{corti} and Theorem \ref{modulos}. Suppose that the result is proved for $m=k$. Then, given
 $X,X_1,\ldots,X_{k+1}\in\matsa$ such that $\|X\|\leq \pi$, let $Y\in\matsa$ be such that $\|Y\|\leq \pi$ and
 $$
 e^{iY}=e^{iX_2}\cdots e^{iX_{k+1}}.
 $$
 By the inductive hypothesis, there exist unitary matrices $V_2,\ldots,V_{k+1}$ such that
 $$
 |Y|\leq \sum_{j=2}^{k+1} V_j|X_j|V_j^*\,.
 $$
 On the other hand, since $e^{iX}=e^{iX_1}e^{iY}$, by the case $n=2$ already proved, there are unitary matrices $U_1$ and $U$ such that
 $|X|\leq U_1|X_1|U_1^*+U|Y|U^*$. If we define $U_j=UV_j$ for $j\geq 2$, then we get the desired result.
\end{proof}

\subsection{The Lagrangians}

Let us list in the following proposition several properties of the symmetric Lagrangian that will be used in the sequel:

\begin{pro}\label{propiedades de L}
Let $\ele:\mat\to[0,\infty)$ be a symmetric Lagrangian, i.e. convex, $\ele(0)=0$, and unitarily invariant in the sense of equation \eqref{L nui}. Then
\begin{enumerate}
\item[(P1)]  $\ele$ is continuous,
\item[(P2)]  $\ele(tA)\leq t\ele(A)$ for every $t\in[0,1]$,
\item[(P3)]  $\ele(A)\leq \ele(B)$ provided $0\leq A\leq B$,
\item[(P4)]  There exists $\phi:\R^n_+\to[0,+\infty)$ such that $\ele(A)=\phi(s(A))$. This $\phi$ is invariant under rearrangement, positive, convex, with $\phi(0)=0$ and $\phi(x)\le\phi(y)$ if $x,y\in \mathbb R_n^+$ and $x_i\le y_i$ for $i=1\dots n$.
\end{enumerate}
\end{pro}
\begin{proof}
The first property is clear because every convex function in a finite dimensional vector space is continuous. Also (P2) is a consequence of the convexity and the fact that $\ele(0)=0$. As $\ele$ is unitarily invariant, the singular value decomposition implies that $\ele(A)$ only depends on the singular values of $A$. Hence, if $x\in\mathbb R_n^+$ and $\mbox{diag}(x)$ denotes the $n\times n$ diagonal matrix whose diagonal entries correspond to the coordinates of $x$, we can define $\phi(x)=\ele(\mbox{diag}(x))$; clearly $\phi(0)=0$, it is non-negative and convex. Convexity implies that if $x,y\in\R^n_+$ and $x_i\leq y_i$ for $i=1,\ldots,n$, then $\phi(x)\leq \phi(y)$. This proves (P4), and (P3) is a direct consequence of it.
\end{proof}

\begin{rem}
 Let $\phi:\R^n_+\to[0,+\infty)$ be a rearrangement invariant, positive and convex function, with $\phi(0)=0$. Then $\phi$ gives place to a symmetric Lagrangian $\ele_\phi$ via the equation $\ele_\phi(A)=\phi(s(A))$. Note that the natural extension of $\phi$ to $\mathbb R^n$ is strongly Schur convex, but not necessarily subadditive.
\end{rem}

\section{Optimality of one parameter subgroups}\label{gs}

A \textit{geodesic segment} is a curve $t\mapsto Ue^{itZ}$ for  $Z\in\matsa$ and  $U\in\matu$. In this section we prove that the geodesic segments (which are parametrized with constant velocity) are optimal for Problem \ref{problema}.  Moreover, if $\ele$ is strictly convex, then we will prove that these geodesic segments are the unique optimal paths.

\subsection{Geodesic segments are short}

\begin{fed}\label{poligonal}
A \textit{polygonal path} is a broken geodesic, that is, a curve $P:[0,b]\to\matu$ such that there is a partition of the interval $[0,b]$ given by the points $0=t_0<$ $\ldots$ $<t_k=b$,  Herminitian matrices $X_1$,$\ldots$,$X_k$ with norm less than or equal to $\pi$, and $U\in\matu$ so that
\begin{equation}\label{eq poli}
P(t)=\begin{cases}
Ue^{i\frac{t}{t_1}X_1}&\mbox{if $t\in[0,t_1]$}\\
Ue^{iX_1}\cdots e^{iX_{j-1}}e^{i\frac{t-t_{j-1}}{t_j-t_{j-1}}X_j}&\mbox{if $t\in[t_{j-1},t_j]$ ($j>1$)}\\
\end{cases}\,.
\end{equation}
\end{fed}

Our first step toward the proof of the optimality of the geodesic segments with constant velocity is the following proposition, which proves that segments are better than polygonal paths. 

\begin{pro}\label{poli}
Let $U\in \matu$ and $V=Ue^{iZ}$, with $Z\in\matsa$ and $\|Z\|\le \pi$. Let $\gamma:[0,b]\to\matu$ be  the segment $\gamma(t)=Ue^{it\frac{Z}{b}}$, and $P:[0,b]\to \matu$  a polygonal path joining $U$ to $V$. Then $\ese(P)\ge \ese(\gamma)$.
\end{pro}
\begin{proof}
Let $0=t_0<$ $\ldots$ $<t_k=b$, and $X_1$,$\ldots$,$X_k\in\matsa$ with norm less than or equal to $\pi$,  so that $P$ has the form showed in \eqref{eq poli} . Then
\begin{align}
\ese(P)&=\sum_{j=1}^{k}\int_{t_{j-1}}^{t_j}\ele\big(\dot{P}(t)\big)\,dt=\sum_{j=1}^{k}\int_{t_{j-1}}^{t_j}\ele\left(\frac{X_j}{t_j-t_{j-1}}\right)\,dt \nonumber\\
&=\sum_{j=1}^{k}(t_j-t_{j-1})\ele\left(\frac{X_j}{t_j-t_{j-1}}\right) \label{cc}
\end{align}
On the other hand, since $e^{iZ}=e^{iX_1}\cdots e^{iX_k}$ and $\|Z\|\leq \pi$, by Proposition \ref{muchas} there exist unitary matrices $U_1,\ldots U_n$ such that
\begin{align}
|Z|\leq \sum_{k=1}^{n} U_k|X_k|U_k^*.\label{no hace falta}
\end{align}
Then, joining \eqref{cc} and \eqref{no hace falta}, and using the properties of $\ele$ we obtain
\begin{align*}
\ese(P)&=b\sum_{j=1}^{k}\frac{(t_j-t_{j-1})}{b}\ele\left(\frac{X_j}{t_j-t_{j-1}}\right) \\&
\geq b\,\ele\left(\frac{1}{b}\sum_{j=1}^{k} U_j|X_j|U_j^*\right)\geq b\,\ele\Big(\frac{Z}{b}\Big)\\
&=\int_0^b\ele\Big(\frac{Z}{b}\Big)\,dt=\ese(\gamma).
\end{align*}

 \end{proof}

To prove that  geodesic segments are optimal paths among all the possible piecewise $C^1$ curves, we need the following standard approximation result by polygonal paths.

\begin{lem}
Let $\alpha:[0,b]\to\matu$ be piecewise smooth. Then for any $\epsilon>0$ there is a polygonal path $P_{\epsilon}:[0,b]\to\matu$ such that for any $t\in [0,b]$,
$$
\|P_{\epsilon}^*(t)\dot{P}_{\epsilon}(t)-\alpha^*(t)\dot{\alpha}(t)\|<\epsilon.
$$
\end{lem}
\begin{proof}
We may as well assume that $\alpha$ is smooth in $[0,b]$. Recall that $\alpha,\dot{\alpha}$ are continuous in the uniform norm. Let $\epsilon>0$, and choose a partition $0=t_0<t_1<\cdots<t_n=b$ of the interval $[0,b]$ such that, for any $k=0,1,\cdots,n$,
$$
\|\alpha(t)-\alpha(s)\|<2\quad\mbox{ and }\quad\|\alpha^*(t)\dot{\alpha}(t)-\alpha^*(s)\dot{\alpha}(s)\|<\frac{\epsilon}{2}
$$
if $s,t\in [t_k,t_{k+1}]$. The first condition implies that there exist $Z_k\in\matsa$ such that $\|Z_k\|<\pi$ and
$e^{iZ_k}=\alpha^*(t_k)\alpha(t_{k+1})$. Moreover, if $\log$ denotes the principal branch of the logarithm, then
$$
Z_k=\log(\alpha^*(t_k)\alpha(t_{k+1})).
$$
Now note that, for any fixed $t\in [0,b]$, the map $g:h\mapsto \frac1h \log(\alpha^*(t)\alpha(t+h))$, is well-defined and analytic, for sufficiently small $h$. Moreover
$$
g(h)\xrightarrow[\ h\to 0\ ]{} \left.\frac{d}{ds}\log \alpha^*(t)\alpha(t+s)\right|_{s=0}=\alpha^*(t)\dot{\alpha}(t).
$$
Then, taking a refinement of the partition if necessary, we can also assume that
$$
\|Z_k-\alpha^*(t_k)\dot{\alpha}(t_k)\|<\frac{\epsilon}{2}
$$
for any $k=0,1,2\cdots, n$. Consider the map $P_{\epsilon}:[0,b]\to\matu$ which is defined as
$$
P_{\epsilon}(t)=\alpha(t_k)e^{\frac{t-t_k}{t_{k+1}-t_k}Z_k}\mbox{ for }t\in [t_k,t_{k+1}].
$$
Then $P_{\epsilon}$ is certainly a polygonal path, and it is straightforward to see that verifies the claim of the lemma.
\end{proof}

\begin{teo}\label{distancia}
Let $U\in \matu$ and $V=Ue^{iZ}$, with $Z\in\matsa$ and $\|Z\|\le \pi$. Then, the curve $\gamma(t)=ue^{itZ/b}$ is optimal among piecewise smooth curves $\alpha:[0,b]\to\matu$ joining $U$ to $V$, with respect to the action $\ese$ defined by a symmetric Lagrangian, and in particular $\inf S=b\ele(Z/b)$.
\end{teo}
\begin{proof}
Given $\epsilon>0$, let $\delta>0$ such that $\|X-Y\|\leq \delta$ implies that $|\ele(X)-\ele(Y)|<\epsilon/b$ for every $X$ and $Y$ in a ball big enough. Then,  let $P_{\delta}$ be a polygonal path in $\matu$ as in the previous lemma, joining $U$ to $V$, such that
$$
\|\dot\alpha-\dot{P}_\delta\|=\|\alpha^*\dot{\alpha}-P_{\delta}^*\dot{P}_{\delta}\|<\delta.
$$
Then by Proposition \ref{poli},
$$
\ese(\gamma)\leq\ese(P_\delta)=\int_0^b \ele(\dot{P}(t))\,dt \leq \eps+ \int_0^b \ele(\dot{\alpha}(t))\,dt<\epsilon+\ese(\alpha),
$$
Therefore, $\ese(\gamma)\leq \ese(\alpha)$.
\end{proof}

\begin{rem}
If $\alpha:[0,b]\to \matu$ is just rectifiable (that is, differentiable $p.p.$ with $\dot{\alpha}(t)$ bounded), the approximation by a polygonal path can be carried out with no major changes, and the proof of the previous theorem shows that in fact, geodesic segments are optimal among {\em rectifiable} arcs joining given endpoints.
\end{rem}

\subsection{Uniqueness of short paths}\label{unicas}

Concerning uniqueness, it is clear that the convexity condition of $\ele$ should be strenghtened.

Let us agree to call $\ele$ {\em nondegenerate} if, given $A,B\in \matsa$, the existence of $\lambda\in (0,1)$ such that the inequality of the convexity condition turns into an equality, implies that there exists $s\ge 0$ such that $A=sB$. In other words, if
$$
\ele(\lambda A+(1-\lambda)B)=\lambda\ele(A)+(1-\lambda)\ele(B)
$$
for some $\lambda\in (0,1)$, then $A=sB$ for some $s\ge 0$. This is a notion of nondegeneracy outside lines. 

The other notion at play here is the strongest notion of {\em strict convexity} of $\ele$, which of course means that if the equality above holds for some $\lambda\in (0,1)$, then $A=B$. A simple example of a strictly convex Lagrangian is the energy functional, given by the square of the Frobenius norm on $\matsa$.

\begin{rem}\label{normastrict}
Note that strict convexity implies nondegeneracy, but the notion of nondegeneracy is relevant since no linear space norm can be strictly convex. In fact, it is usual to say that a norm $\|\cdot\|$ on a linear space is strictly convex when the weaker condition (nondegeneracy) stated above holds, which due to the homogeneity of the norm amounts to say that
$$
\|A+B\|=\|A\|+\|B\|
$$
implies $A=sB$ for some $s\ge 0$, and geometrically, is equivalent to the fact that the unit ball of the normed space has no segments.
\end{rem}

We begin with a technical lemma. Recall that if $A\in\matsa$, then $\avi{1}{A}$, $\ldots$, $\avi{n}{A}$ denotes the eigenvalues of $A$ arranged in non-increasing way.

\begin{lem}\label{lanza la bola chico}
Let  $X,Y,Z\in\matsa$ be such that $e^{Z}=e^{iX}e^{iY}$ and $\|Z\|<\pi$. If  $\avi{k}{X}=r \avi{k}{Z}$ and $\avi{k}{Y}=(1-r) \avi{k}{Z}$  for some $r \in [0,1]$ and every $k\in\{1,\ldots,n\}$, then $X=r Z$ and $Y=(1-r) Z$.
\end{lem}
\begin{proof}
It is enough to show that $Z$ shares an orthonormal basis of eigenvalues with $X$ and $Y$. Let $\xi$ be an unitary eigenvector of $Z$  such that $|Z|\xi=\|Z\|\xi$. Consider the unit sphere $S^{n-1}\subset \mathbb C^n$ and the maps $\alpha,\beta:[0,1]\to S^{n-1}$ given by $\alpha(t)=e^{itZ}\xi$,
$$
\beta(t)=\left\{\begin{array}{rr}
e^{2itX}\xi &\mbox{ if }t\in[0,1/2]\\
e^{iX}e^{2i(t-1/2)Y}\xi & \mbox{ if } t\in[1/2,1]\end{array}\right.\, .
$$
In particular, $\alpha$ and $\beta$ have the same extreme points. A simple computation shows that, with respect to the natural Riemannian structure, $\mbox{Long} (\alpha)=\mu$ and $\mbox{Long} (\beta)\leq \mu$. But, since
$$
\ddot{\alpha}(t)=e^{itZ}(-Z^{\,2})\xi=-e^{itZ}|Z|^2\xi=-\|Z\|^2 e^{itZ}\xi=-\|Z\|^2 \alpha(t)
$$
and $\mbox{Long}(\alpha)=\|Z\|< \pi$, then $\alpha$ is the unique short geodesic of the sphere $S^{n-1}$ joining $\xi$ with $e^{iZ}\xi$. So, $\mbox{Graph}(\alpha)=\mbox{Graph}(\beta)$ and $\xi$ is also an eigenvalue of $X$ and $Y$. Iterating this procedure, we can conclude that $X$, $Y$ and $Z$ share a common orthonormal basis of eigenvalues.
\end{proof}

\begin{teo}\label{cambiadito}
Assume that $\ele$ is strictly convex. Let  $X,Y\in\matsa$ with norm less or equal than $\pi$, and $Z\in\matsa$ such that $\|Z\|<\pi$ and $e^{iZ}=e^{iX}e^{iY}$. Consider the geodesic segment $\gamma:[0,b]\to\matu$ defined by $\gamma(t)=e^{itZ/b}$, and the polygonal $P:[0,1]\to\matu$defined by
$$
\begin{cases}
e^{i\frac{t}{t_0}X}&\mbox{if $t\in[0,t_0]$}\\
e^{iX}e^{i\frac{t-t_0}{b-t_0}Y}&\mbox{if $t\in[t_0,b]$}\\
\end{cases}\,.
$$
for some $t_0\in(0,b)$. If $\ese(P)=\ese(\gamma)$ then $X=\frac{t_0}{b} Z$ and $P=\gamma$.
\end{teo}
\begin{proof}
By Proposition \ref{porron}, there exist unitary matrices $U$ and $V$ such that
$$
e^{iZ}=e^{i(UXU^*+VYV^*)} \peso{and} |Z|\leq |UXU^*+VYV^*|\,,
$$
and by the computations made in Proposition \ref{poli} (Equation \eqref{cc})
$$
\ese(P)=t_0\,\ele\left(\frac{X}{t_0}\right)+(b-t_0)\,\ele\left(\frac{Y}{b-t_0}\right)\,.
$$
Then, using the properties of $\ele$, the hypothesis $\ese(P)=\ese(\gamma)$ implies that
\begin{align*}
\ese(\gamma)&=\ese(P)=t_0\,\ele\left(\frac{X}{t_0}\right)+(b-t_0)\,\ele\left(\frac{Y}{b-t_0}\right)\\
           &=b\left(\frac{t_0}{b}\,\ele\left(\frac{UXU^*}{t_0}\right)+\frac{b-t_0}{b}\,\ele\left(\frac{VYV^*}{b-t_0}\right)\right)\\
           &\geq b\ele\left(\frac{UXU^*+VYV^*}{b}\right)\geq b\,\ele\Big(\frac{Z}{b}\Big)
           =\ese(\gamma).
\end{align*}
On one hand, this implies that $Z=UXU^*+VYV^*$.  Indeed, if $W=UXU^*+VYV^*$ then $|Z|\leq|W|$. But the above chain of identities implies that $\ele(Z)=\ele(W)$, and (P2) in Proposition \ref{propiedades de L} implies that $|Z|=|W|$. Hence, $0\leq |Z|=|W|<\pi$. Since $e^{iZ}=e^{i(UXU^*+VYV^*)}$ we get the desired equality. On the other hand, since $\ele$ is strictly convex if $r=t_0/b$ then
\begin{align*}
rZ=UXU^* \peso{and} (1-r)Z=VYV^*.
\end{align*}
Now, by Lemma \ref{lanza la bola chico} we obtain that $X=UXU^*$ and $Y=VYV^*$ which concludes the proof.
\end{proof}

\begin{teo}\label{esunica}
Assume that $\ele$ is strictly convex. Let $Z\in\matsa$ be such that $\|Z\|<\pi$. Then, the geodesic segment $\delta:[0,b]\to\matu$ defined by $\gamma(t)=Ue^{itZ/b}$ is the unique piecewise $C^1$ curve in $\matu$ joining $U$ to $V=Ue^{iZ}$, and $\ese(\delta)=b\ele(Z/b)$.
\end{teo}
\begin{proof}
Without lost of generality we can assume that $U=1$. Suppose that $\alpha$ is any short, piecewise smooth curve joining $1$ to $e^{iZ}$. Let $t_0\in (0,1)$ and let $\alpha(t_0)=e^{iX}=e^{iZ}e^{-iY}$, with $\|Y\|\le \pi$, $\|X\|\le \pi$. Consider the polygonal $P:[0,b]\to\matu$ defined by
$$
\begin{cases}
e^{i\frac{t}{t_0}X}&\mbox{if $t\in[0,t_0]$}\\
e^{iX}e^{i\frac{t-t_0}{b-t_0}Y}&\mbox{if $t\in[t_0,b]$}\\
\end{cases}\,.
$$
Then, by Proposition \ref{poli} and Theorem \ref{distancia} applied to each segment,
\begin{align*}
\ese(\gamma)&\le \ese(P)\leq \int_0^{t_0}\ele(\dot{\alpha})\,dt +\int_{t_0}^b\ele(\dot{\alpha})\,dt=\ese(\alpha)=\ese(\gamma),
\end{align*}
Hence $\ese(\gamma)=\ese(P)$, and by Theorem \ref{cambiadito} we get that $X=\frac{t_0}{b}Z$.
\end{proof}

This settles Problem \ref{problema2} when the Lagrangian is strictly convex: the geodesic segments are optimal and unique {as functions}. Regarding the second question of that problem, we have the following result, that settles this poblem when the Lagrangian is nondengenerate (for instance, if $\ele$ is a strictly convex norm on a linear space, Remark \ref{normastrict}): in this case, geodesic segments are optimal and unique modulo a reparametrization of the path, that is, they are unique in a geometrical sense.

\begin{teo}\label{esunicasalvo}
Assume that $\ele$ is nondegenerate. Let $Z\in\matsa$ be such that $\|Z\|<\pi$. Then, if $\alpha:[0,b]\to\matu$ is an optimal path of the minimization problem given by $\ele$ with given endpoints $U,V$, $\alpha$ must be a reparametrization of the geodesic segment $\gamma:[0,b]\to\matu$ defined by $\gamma(t)=Ue^{itZ/b}$.
\end{teo}
\begin{proof}
We assume that $U=1$ and $V=e^{iZ}$. Let $t_0\in (0,1)$ and let $\alpha(t_0)=e^{iX}=e^{iZ}e^{-iY}$, with $\|Y\|\le \pi$, $\|X\|\le \pi$. Arguing as in the proof of Theorem \ref{cambiadito}, convexity of $\ele$ and minimality of $\alpha$ imply that $Z=UXU^*+VYV^*$. Now, nondegeneracy of $\ele$ implies also that there exists $s\ge 0$ such that
$$
\frac{UXU^*}{t_0}=s\frac{VYV^*}{b-t_0}.
$$
Now we take $s_0=\frac{st_0}{b-t_0}\ge 0$ and $r=(1+s_0)^{-1}$. Note that $r\in [0,1]$ and also that $rZ=UXU^*$, $(1-r)Z=VYV^*$. Invoking once again Lemma  \ref{lanza la bola chico}, it follows that $X=UXU^*$, $Y=VYV^*$. Thus $\alpha(t_0)=e^{irZ}$ and then $\alpha$ must be a reparametrization of the geodesic segment $\gamma$.
\end{proof}

Regarding uniqueness of paths when $\|U-V\|=2$ (or equivalently, when $V=Ue^{iZ}$ and $\|Z\|=\pi$), this property is not expected since taking $n=1$, $U=1$, $V=-1$ shows that there are two geodesic segments in the circumference ($=\mathcal U(1)$) joining $U,V$, and the situation worsens as $n$ gets bigger.

\section{Rectifiable distances in $\matu$ and angular metrics in the Grassmann manifold}

In this section, we focus in the particular case where $\ele$ is a unitarily invariant norm. In that case the action $\ese$ defines a length of curves and the length of the optimal path defines a distance in $\matu$. 

\subsection{Unitarily invariant norms and symmetric gauge functions}\label{nuis}

One of the most relevant properties of the uniform norm of matrices is the following: given two unitary matrices $U$ and $V$, then $\|UTV\|=\|T\|$. This property is shared by many other norms defined in $\mat$.

\begin{fed}\label{nui def}
A norm $\nnui{\cdot}$ defined in $\mat$ is called unitarily invariant if for every matrix $T$ and every pair of unitary matrices $U$ and $V$  it holds that
$\nnui{UTV}=\nnui{T}$.
\end{fed}

As a consequence of the singular value decomposition, $\nnui{T}=\nnui{|T|}$, and
\begin{equation}\label{formulagn}
\nnui{T}=\|T\|_\phi=\phi(s(T))\,,
\end{equation}
where $\phi$ is a {\it symmetric gauge function}, that is, a rearrangement invariant norm on $\mathbb R^n$, and depends only on the moduli of the coordinates of the vectors. The next theorem \cite{bhatia} will be useful in what follows:
\begin{teo}\label{nui gsf}
There is a bijection bewtween symmetric gauge functions $\phi$ on $\R^n$, and unitarily invariant norms $\|\cdot\|_{\phi}$ on $\mat$  given by equation (\ref{formulagn}) above.
\end{teo}

\subsection{Rectifiable metrics in the unitary group}\label{distancias}

By considering as a Lagrangian a unitarily invariant norm $\|\cdot\|_{\phi}$, the action $S$ can be interpreted as the length of curves $L_{\phi}$, and the rectifiable distance between $U,V\in \matu$ is
$$
d_{\phi}(U,V)=\inf\left\{ L_{\phi}(\gamma) |\, \gamma:[a,b]\to\matu\mbox{ is piecewise smooth and joins } U \mbox{ to } V \mbox{ in }\matu\right\}.
$$

The function $d_{\phi}$ is in fact a distance, since  $\|U-V\|_{\phi}\simeq d_{\phi}(U,V)$ for  any $U,V\in \matu$. One of the main features of this metric is that it is invariant for the action of the unitary group $\matu$, in fact it is a bi-invariant metric
$$
d_{\phi}(UV_1W,UV_2W)=d_{\phi}(V_1,V_2)
$$
for $U,W,V_1,V_2\in \matu$.

\subsubsection{Minimality of one-parameter subgroups}

 As a direct consequence of Theorem \ref{distancia} and Theorem \ref{esunicasalvo}, we obtain the following result, which  generalizes  \cite[Theorem 3.2]{upe} for the $p$-norms ($p\ge 2$), see also \cite{mlr}.

\begin{teo}\label{distancia nui}
Let $U,V\in \matu$ and $V=Ue^{iZ}$, with $\|Z\|\le \pi$, $Z\in \matsa$. Then, the curve $\delta(t)=Ue^{itZ}$ is shorter than any other piecewise smooth curve $\gamma$ in $\matu$ joining $U$ to $V$, when we measure them with the norm $\nf$. In particular, $d_{\phi}(U,V)=\|Z\|_{\phi}$. If $\|U-V\|<1$ (equivalently, if $\|Z\|<\pi$), then this $\delta$ is the unique short path joining $U,V$ in $\matu$ provided the norm is stricly convex.
\end{teo}

\begin{rem}\label{adicione}
 A question related to the uniqueness of geodesics, is if we can ensure that the points in $\matu$ are aligned when the distance is additive. That is, if
$$
d_{\phi}(U,V)=d_{\phi}(U,W)+d_{\phi}(W,V).
$$
implies that there exists $t_0\in [0,1]$ and $X_0\in\matsa$ with $\|X_0\|\le \pi$ such that
$$
V=Ue^{iX_0},\qquad \peso{while} W=Ue^{it_0 X_0}.
$$
\end{rem}

The previous theorem implies this when $\|U-V\|<2$. However, the question always has an affirmative answer (provided the norm is strictly convex), with a simpler proof.

\begin{teo}\label{cambiadito2}
Assume that the norm $\nf$ is strictly convex, and let $U,V,W\in\matu$ be such that
$$
d_{\phi}(U,V)=d_{\phi}(U,W)+d_{\phi}(W,V).
$$
Then $U,V,W$ are aligned in $\matu$.
\end{teo}
\begin{proof}
We can assume that $U=1$, $V=e^{iZ}$, $W=e^{iX}$ with $X,Z$ of norm less or equal than $\pi$. Let $Y\in\matsa$ such that $\|Y\|\le \pi$ and $e^{iZ}=e^{iX}e^{iY}$. Then the hypothesis is that
$$
\|Z\|_\phi=\|X\|_\phi+\|Y\|_\phi.
$$
Consider the smooth path $\alpha(t)=e^{itX}e^{itY}$. Then $\alpha$ joins the same endpoints that $\delta(t)=e^{itZ}$ in $\matu$, thus 
$$
\|X+Y\|_\phi=L_{\phi}(\alpha)\ge L_\phi(\delta)=\|Z\|_\phi=\|X\|_\phi+\|Y\|_\phi.
$$
Since the norm is strictly convex, there exists $\lambda\ge 0$ such that $Y=\lambda X$. Pick $X_0=(1+\lambda)X$ and $t_0=(1+\lambda)^{-1}$ to finish the proof.
\end{proof}

\subsection{The Grassmannian}\label{soyo}

The Grassmannian $\cG_n$ is the set of subspaces of $\C^n$, which can be identified with the set of orthogonal projections in $\mat$. If we consider in $\mat$ the topology defined by any of all the equivalent norms,  the Grassmann space endowed with the inherited topology  becomes a compact set. However, it is not connected. Indeed, it is enough to consider the trace $tr$, which is a continuous map defined on the whole space $\mat$, and restricted to $\cG_n$ takes only positive integer values. In particular, this shows that the connected components of $\G_n$ are the subsets $\G_{m,n}$ defined as:
$$
\G_{m,n}:=\{P\in\G_n:\ \tr(P)=m\}.
$$
Each of these components is a submanifold of $\mat$ \cite[p.129]{W}, and connected components are given by the unitary orbit of a given projection $P$ such that $tr(P)=m$:
$$
\G_{m,n}=\{UPU^*:U\in \matu\}.
$$
The tangent space at a point $P\in\G_{m,n}$ can be identified with the subspace of $P$-codiagonal Hermitian matrices, i.e.
$$
T_P\G_n=\left\{X\in\matsa:\ X=PX+XP\right\}\,.
$$
In particular note that $T_P\G_n$ has a natural complement $\sub{N}{P}$, which is the space of Hermitian matrices that commute with $P$, that is, the $P$-diagonal Hermitian matrices. The decomposition in diagonal and codiagonal matrices defines a normal bundle, and leads to a covariant derivative
\begin{align}
\nabla_V \,\Gamma(P) =\sub{\Pi}{\sub{T}{P}||\sub{N}{P}}\left.\frac{d}{dt}\Gamma(\alpha(t))\right|_{t=0}\,, \label{BebiCepita}
\end{align}
where $\Gamma$ is a vector field along the curve $\alpha:(-\eps,\eps)\to\cG_{m,n}$ that satisfies $\alpha(0)=P$ and $\dot\alpha(0)=V$. So, we have a notion of parallelism, and the geodesics in this sense are described by the following theorem:
\begin{teo}[Porta-Recht \cite{pr}]\label{geodesics}
The unique geodesic at $P$ with direction $X$ is:
$$
\gamma(t)=e^{\,itX}Pe^{\,-itX}\,.
$$
\end{teo}

As the unitary group acts transitively in these components via $U\cdot P=UPU^*$, they are also homogeneous spaces of $\matu$. They can be distinguished from other homogeneous submanifolds of $\matu$, because the map
$$
P\mapsto \sub{S}{P}=2P-1
$$
embeds them in $\matu$, and the map $S$ is two times an isometry. The images $S_P$ are symmetries, i.e. matrices that satisfy  $\sub{S}{P}^*=\sub{S}{P}=\sub{S}{P}^{-1}$.

\subsubsection{Finsler metrics on the Grassmannian}

For a given symmetric norm, the Grassmann space carries the Finsler structure given by
$$
\|X\|_P=\|X\|_{\phi}
$$
for $X\in T_P\G_n$, and with this structure, the Grassmann component $\{UPU^*:U\in \matu\}$ is isometric (modulo a factor $2$) to the orbit of symmetries $\{US_PU^*: U\in \matu\}$.  In the particular case when $\|\cdot\|_{\phi}$ is the Frobenius norm, this connection is the Levi-Civita connection of the metric, since the $P$-diagonal matrices are the orthogonal complement of the $P$-codiagonal matrices with respect to this Riemannian metric.

\medskip

A straightworward computation shows that, if $X=XP+PX$, then $e^{iX}S_P=S_Pe^{-iX}$. This simple observation enables to use our results in the unitary group, to prove minimality of geodesics in the Grassmann manifold:
\begin{teo}\label{grasita}
If $P,Q\in \G_{m.n}$ then there exists $X\in T_P\G_n$ such that $Q=e^{iX}Pe^{-iX}$ and $\|X\|\le\frac{\pi}{2}$, unique when $\|P-Q\|<1$. The geodesic $\gamma(t)=e^{itX}Pe^{-itX}$ is shorter than any rectifiable path in $\G_n$ joining $P,Q$ and 
$$
d_{\phi}(P,Q)=\|XP-PX\|_{\phi}=\|X\|_{\phi}.
$$
If the norm is strictly convex and $\|P-Q\|<1$, the geodesic is the unique short path joining $P,Q\in \G_n$.
\end{teo}
\begin{proof}
The existence of $X$ follows from Halmos \cite{halmos} or Davis and Kahan \cite{[DK]}. Since $e^{2iX}=S_QS_P$, if $\|Q-P\|<1$ this $X$ is unique. Since
$$
S_{\gamma(t)}=2\gamma(t)-1=e^{itX}S_Pe^{-itX}=e^{2itX}S_P=S_Pe^{-2itX},
$$
and $S$ is two times an isometry, the minimality of $\gamma$ follows from Theorem \ref{distancia nui}, and the same applies to the uniqueness in the strictly convex case. Finally, $L_{\phi}(\gamma)=\|XP-PX\|_{\phi}$, and on the other hand, since $PXP=0$ then
$$
|XP-PX|^2=|XP+PX|^2=|X|^2,
$$
thus $d_{\phi}(P,Q)=L_{\phi}(\gamma)=\| |XP-PX| \|_{\phi}= \| |X| \|_{\phi}=\|X\|_{\phi}$.
\end{proof}

\begin{rem}\label{biendefi}
In the situation of the previous theorem, it is not hard to see that if $k\in\mathbb Z$, then $PX^{2k}=X^{2k}P$, $PX^{2k+1}=-PX^{2k+1}$. Then $P|X|=|X|P=|XP|$ and $(1-P)|X|=|X|(1-P)=|PX|$. Moreover 
$$
Q=P\cos^2 X + (1-P)\sin^2 X - \frac{i}{2}P\sin 2X+ \frac{i}{2}(1-P)\sin 2X,
$$
and then $|PQ|^2=PQP=P\cos^2 X$, which leads to $|PQ|=P\cos X=\cos |XP|$, and likewise $|QP|=(1-P)\cos X=\cos |PX|$. Thus if $Y\in T_p\G_n$ is any other matrix as $X$, it follows that $P\cos X=P\cos Y$ or equivalently, 
$$
\cos |XP|=|PQ|=\cos |YP|.
$$
\end{rem}

\subsection{The angular metrics}

Let $\eme$ and $\ete$ be two $m$-dimensional subspaces of $\C^n$, and let $\sub{P}{\eme}$ and $\sub{P}{\ete}$ be the orthogonal projections onto $\eme$ and $\ete$ respectively. \textit{The principal angles} between $\eme$ and $\ete$ are the angles $\theta_1(\eme,\ete),\ldots,\theta_m(\eme,\ete)\in[0,\pi/2)$ whose cosines are the $m$ greatest singular values of $\sub{P}{\eme}\sub{P}{\ete}$, see  \cite{[J]}.

In \cite{zhang1} Li, Qiu, and Zhang used the principal angles to define metrics in the components of $\G_{m,n}$. Given a symmetric norm $\|\cdot\|_\phi$, they define for $P,Q\in \cG_{m,n}$ the following distance:
$$
\rho_\phi(P,Q)=\|\arccos |PQ|\|_\phi .
$$
These distances are called \textit{angular metrics}, because if $\phi$ is the symmetric gauge function associated to $\|\cdot\|_\phi$ then
$$
\rho_\phi(P,Q)=\phi(\theta_1(\eme,\ete),\ldots,\theta_m(\eme,\ete),0,\ldots,0).
$$
where $\eme=R(P)$ and $\ete=R(Q)$. The definition of these metrics was motivated not only by pure mathematics but also by engineering applications. For example, in robust control, a linear time-invariant system can be described by a subspace valued frequency function, and the description of an uncertain system needs a suitable distance measure between subspaces. The reader is referred to \cite{zhang1}, where other motivations and applications of these metrics are described.

\medskip

A legitimate question at this point, is if these distances are related to an infinitesimal structure on the manifold $\G_n$, that is, if the angular distance among $P,Q\in\G_{m,n}$ can be computed as the infima of the lengths of the rectifiable arcs joining $P,Q$. Note that, by Remark \ref{biendefi}, if $X$ is as in Theorem \ref{grasita}, then the angular distance among $P,Q$ can be computed as
$$
\rho_\phi(P,Q)=\|\arccos |PQ|\|_\phi =\|XP\|_\phi
$$
and this computation does not depend on the particular $X$. Then, one can be tempted to endow the Grassmannian with the Finsler metric (i.e. tangent norm) given by $\|X\|_P=\|XP\|_{\phi}$ for $X\in T_P\G_n$. The problem with this definition is that it is not clear how to extended it to the whole $\mat$ in order to obtain an unitarily invariant norm there. 

\medskip

To this end, it suffices to consider the case $m\leq n/2$. Let $\phi$ be the symmetric gauge function associated to $\|\cdot\|_\phi$ (see Theorem \ref{nui gsf}), and define $\|\cdot\|_\psi$ in the following way:
\begin{equation}\label{recti}
\|A\|_\psi=\phi\big(1/2(\svi{1}{A}+\svi{2}{A},\ldots,\svi{2m-1}{A}+\svi{2m}{A},0,\ldots,0)\big)\,,
\end{equation}
where $\svi{1}{A}$,$\ldots$,$\svi{n}{A}$ denotes the singular values of $A$ counted with multiplicity and ordered in non-increasing way\footnote{The arithmetic mean can be replaced by any positive mean.}. Straightforward computations show that $\|\cdot\|_\psi$ is a symmetric norm, and also that, for any $Q \in \G_{m,n}$ and $Z\in T_Q\G_n$ it holds
$$
\|QZ\|_\phi=\|Z\|_\psi.
$$

The following theorem gives the link between the rectifiable distances and the angular metrics:
\begin{teo}[Davis-Kahan \cite{[DK]}] \label{direct rotations}
Let $P,Q\in\cG_{m,n}$, and denote $\eme=R(P)$ and $\ete=R(Q)$. Then, if  $X\in\matsa$ is $P$-codiagonal with $\|X\|\le\pi/2$ and $Q=e^{iX}Pe^{-iX}$, its spectrum counted with multiplicity is
$$
\big(\pm\theta_1(\eme,\ete),\ldots,\pm\theta_m(\eme,\ete),0\ldots,0\big).
$$
\end{teo}

Consider the rectifiable distance $d_\psi$ associated to the norm given in (\ref{recti}), and take $P,Q,X$ as in Theorem \ref{grasita}. Then
\begin{align*}
d_\psi(P,Q)&=\|X\|_\psi=\phi\big(1/2(\svi{1}{X}+\svi{2}{X},\ldots,\svi{2m-1}{X}+\svi{2m}{X},0,\ldots,0)\big)\\
&=\phi\big(\theta_1(\eme,\ete),\ldots,\theta_m(\eme,\ete),0\ldots,0\big)\\
&=\rho_\phi(P,Q)\,,
\end{align*}
by Theorem \ref{direct rotations}, and this establishes the following (obtained by Neretin in \cite{rusohdp} with another proof):
\begin{teo}\label{son iguales}
Let $\|\cdot\|_\phi$ be a symmetric norm, and $\rho_\phi$ its corresponding angular metric in $\cG_{m,n}$. Then, there exists an induced symmetric norm $\|\cdot\|_\psi$ such that the corresponding rectifiable distance $d_\psi$ coincides with $\rho_\phi$.
\end{teo}

\begin{rem}
In \cite[Section 4]{zhang1}, the authors prove that when the norm $\|\cdot\|_\phi$ is strictly convex, if the distance among $P,Q,R\in\G_{m,n}$ is additive, then there exists a direct rotation from $\eme$ to $\eze$ through $\ete$, where $\eme=R(P),\ete=R(Q)$ and $\eze=R(R)$. This last assertion is equivalent to the notion of being aligned as introduced in Remark \ref{adicione}. Thus the proof of this fact follows immediatly from Theorem \ref{cambiadito2}.
\end{rem}

\appendix{
\section{Appendix: compact operators}

The results of the previous sections can be extended to the infinite dimensional setting as follows. Let $\h$ be a complex separable Hilbert space, $\bh$ the algebra of bounded  operators with the supremum norm, $\kh$ the algebra of compact operators, $\u$ the group of unitary operators. Let $\|\cdot\|_{\phi}:\bh\to \mathbb R\cup \{\infty\}$ be a {\em symmetric norm}, that is a norm such that
\begin{equation}\label{syn}
\|AXB\|_{\phi}\le \|A\|\|X\|_{\phi}\|B\|
\end{equation}
for $A,X,B\in \bh$ (both sides can equal $\infty$). In particular, it is unitarily invariant, thus it only depends on the singular values of the operator, and as in Theorem \ref{nui gsf}, there is a symmetric gauge function $\phi:\mathbb R^{\infty}\to \mathbb R_{\ge 0}$  related to this norm; the relationship is somewhat subtle so we refer the reader to Simon's book \cite{simon} for full details on these {\em symmetrically normed ideals}. 

Let ${\cal I}\subset \kh$ stand for the ideal of operators with finite norm, which will be assumed to be complete with respect to its norm, and let $\ufi=\{u\in \u: u-1\in {\cal I}\}$. This  is a Banach-Lie group, whose Banach-Lie algebra can be readily identified with the anti-Hermitian part of ${\cal I}$, that we will denote with $i{\cal I}_{h}$. A straightforward computation using the functional calculus and the fact that $\cal I$ is an ideal shows that if $\|Z\|\le \pi$ is self-adjoint and $e^{iZ}=U$, then $Z\in \cal I$.

\subsection{The special unitary groups}

The length functional on $\ufi$ is defined accordingly as $L_{\phi}(\alpha)=\int_0^1 \|\dot{\alpha}\|_{\phi}$, and the distance $d_{\phi}$ is defined as the infima of the lengths of curves in $\ufi$ joining given endpoints; in order to prove minimality of geodesic segments, we will need the following extension of Thompson's formula, its proof can be found in \cite[Theorem 3.2]{alv}:
\begin{teo}\label{ese}
Given $X, Y\in \kh_h$, there is an isometry $w\in\op$ ($w^*w=1$), and unitary operators $U$ and $V$ such that
$$
e^{i\, wXw^*}e^{i\, wYw^*}=e^{i\,  U(wXw^*)U^*+i\, V(wYw^*)V^*}\,.
$$
\end{teo}

\begin{teo}
Let $U,V\in \ufi$, $Z\in {\cal I}$ such that $V=Ue^{iZ}$ and $\|Z\|\le \pi$. Then, the curve $\gamma(t)=Ue^{itZ}$ is minimal among rectifiable curves $\alpha\subset \ufi$ joining $U,V$, with respect to the distance induced by the length $L_{\phi}$, and $d_{\phi}(U,V)=\|Z\|_{\phi}$. This curve is unique if the norm is strictly convex and $\|U-V\|<2$ (equivalently, $\|Z\|<\pi$).
\end{teo}
\begin{proof}
If $Z\in \cal I$ is such that $e^{iZ}=e^{iX}e^{iY}$ and $\|Z\|\le \pi$ (where we can assume that $X,Y\in\cal I$), then  $e^{iwZw^*}=e^{iwXw^*}e^{iwYw^*}$ for some isometry $w\in\bh$ by Theorem \ref{ese}. With the same proof as Corollary  \ref{corti}, we obtain
$$
|wZw^*|\le |U(wXw^*)U^*+i\, V(wYw^*)V^*|.
$$
Due to (\ref{syn}), it follows that
$$
\|Z\|_{\phi}=\|w^*wZw^*w\|_{\phi}\le \|wZw^*\|_{\phi}\le \|X\|_{\phi}+\|Y\|_{\phi}
$$
since $w$ is an isometry thus $\|w\|=1$. Now the rest of the proof of minimality of segments follows as in Section \ref{gs}. The uniqueness when the norm is strictly convex can be proved  invoking Theorem \ref{ese}, and arguing as in the proof of Theorem \ref{esunicasalvo}.
\end{proof}

\subsection{The restricted Grassmannians}

The same considerations hold for the special Grassmannian manifold, whose components can be regarded as unitary orbits of self-adjoint projections $P\in\bh$, with the action of these special unitary groups:
$$
\G_{\phi}(P)=\{UPU^*: U\in \ufi\}. 
$$
Since $U-1\in \cal I$, then the orbit is contained in the affine space $P+\cal I$. Then tangent spaces are identified with
$$
T_P\G_{\phi}(P)=\{X\in {\cal I}_h: XP+PX=X\}.
$$
A well-known result of Halmos \cite{halmos} says that if $P,Q\in \bh$ are self-adjoint projections whose ranges have the same dimension (including the posiblity of $+\infty$), and the same holds for their kernels, then there exists a $P$-codiagonal $X$ such that $\|X\|\le \frac{\pi}{2}$ and $Q=e^{iX}Pe^{-iX}$. Since $\G_{\phi}\subset P+\cal I$, it is easy to check that $S_QS_P\in \ufi$. Then, $e^{2iX}=S_QS_P$ is also in $\ufi$, and it follows that $X\in \cal I$. 
\begin{cor}
If $P,Q\in \G_{\phi}(P)$ then there exists $X\in T_P\G_{\phi}(P)$ such that $Q=e^{iX}Pe^{-iX}$ and $\|X\|\le\frac{\pi}{2}$, unique when $\|P-Q\|<1$. The geodesic $\gamma(t)=e^{itX}Pe^{-itx}$ is shorter than any rectifiable path in $\G_{\phi}(P)$ joining $P,Q$ and  $d_{\phi}(P,Q)=\|XP-PX\|_{\phi}=\|X\|_{\phi}$. If the norm is strictly convex and $\|P-Q\|<1$, the geodesic is the unique short path joining $P,Q\in \G_{\phi}(P)$.
\end{cor}

\begin{rem}
When ${\cal I}$ is the ideal of Hilbert-Schmidt operators, the special Grassmannian defined above is known as the {\em Sato Grassmannian} or the {\em restricted Grassmannian}. The proof of minimality of one-parameter groups in this Riemann-Hilbert setting was given in \cite{al} with a different technique.
\end{rem}

}

\bigskip

\noindent

\hspace*{.4cm}\begin{minipage}{8cm}
Jorge Antezana:\\
Universitat Aut\'onoma de Barcelona.\\
Departamento de Matemática,\\
Facultad de Ciencias\\
Edificio C Bellaterra (08193) \\
Barcelona, Espa\~{n}a.\\
e-mail: jaantezana@mat.uab.cat\\

\medskip

Gabriel Larotonda and Alejandro Varela:\\
Instituto de Ciencias \\
Universidad Nacional de General Sarmiento. \\
J. M. Gutiérrez 1150 \\
(B1613GSX) Los Polvorines, \\
Buenos Aires, Argentina.  \\
e-mails: glaroton@ungs.edu.ar,\\
avarela@ungs.edu.ar\\

\end{minipage}
\hspace*{.1cm}
 \begin{minipage}{6cm}
 \vspace*{-4.7cm}
 J. Antezana, G. Larotonda\\
 and A. Varela:\\
 Instituto Argentino de Matemática\\
 ``Alberto P. Calder\'on'', CONICET\\
 Saavedra 15, 3er piso\\
 (C1083ACA) Buenos Aires,\\
 Argentina.\\

 \end{minipage}


\begin{thebibliography}{XX}


\bibitem{al} E. Andruchow, G. Larotonda, {\it Hopf-Rinow theorem in the Sato Grassmannian}. J. Funct. Anal. 255 (2008), no. 7, 1692--1712.

\bibitem{alv} J. Antezana, G. Larotonda, A. Varela, {\it Thompson-type formulae}, preprint 	arXiv : 1107.0348v1 (2011).

\bibitem{andcan} E. Andruchow, {\it Short geodesics of unitaries in the $L^2$ metric}. Canad. Math. Bull. 48 (2005), no. 3, 340--354.

\bibitem{upe}  E. Andruchow, G. Larotonda, L. Recht, {\it Finsler geometry and actions of the $p$-Schatten unitary groups}, Trans. Amer. Math. Soc. 62 (2010), 319-344.

\bibitem{bhatia} R. Bhatia. Matrix analysis. Graduate Texts in Mathematics, 169. Springer-Verlag, New York, 1997.

\bibitem{[DK]} C. Davis, W. M.  Kahan, {\it The rotation of eigenvectors by a perturbation. III}, SIAM J. Numer. Anal. 7 (1970) 1--46.

\bibitem{gohberg} I. C. Gohberg, M. G. Krein.  Introduction to the theory of linear nonselfadjoint operators. Translated from the Russian by A. Feinstein. Translations of Mathematical Monographs, Vol. 18 American Mathematical Society, Providence, R.I. 1969.

\bibitem{halmos} P. R. Halmos, {\it Two subspaces}. Trans. Amer. Math. Soc. 144 (1969) 381--389.

\bibitem{[J]} C. Jordan, Essai sur la géométrie à $n$ dimensions, Bull. Soc. Math. France, 3 (1875), pp. 103--174.

\bibitem{zhang1} C.-K. Li, L. Qiu, Y. Zhang.  {\it Unitarily invariant metrics on the Grassmann space}. SIAM J. Matrix Anal. Appl.  27  (2005), no. 2, 507--531 (electronic).

\bibitem{mlr} L.E. Mata-Lorenzo, L. Recht. {\it Convexity properties of ${\rm Tr}[(a\sp *a)\sp n]$}. Linear Algebra Appl. 315 (2000), no. 1-3, 25--38.

\bibitem{rusohdp} Y. A. Neretin, {\it On Jordan angles and the triangle inequality in Grassmann manifolds}, Geom. Dedicata 86 (2001), 81--92.

\bibitem{pr} H. Porta, L. Recht, {\it Minimality of geodesics in Grassmann manifolds}, Proc. Amer. Math. Soc. 100 (1987), no. 3, 464--466.

\bibitem{simon} B. Simon. Trace ideals and their applications. Second edition. Mathematical Surveys and Monographs, 120. American Mathematical Society, Providence, RI, 2005.

\bibitem{[T1]} R. C. Thompson, {\it Convex and concave functions of singular values of matrix sums}, Pacific J. Math. 66 (1976), no. 1, 285--290.

\bibitem{[T2]} R. C. Thompson, {\it Matrix type metric inequalities}, Linear and Multilinear Algebra 5 (1977/78), no. 4, 303--319.

\bibitem{thompson} R.C. Thompson, {\it Proof of a conjectured exponential formula}. Linear and Multilinear Algebra  19  (1986), no. 2, 187--197.

\bibitem{W} F. W. Warner, Foundations of differentiable manifolds and Lie groups, Graduate Texts in Mathematics 94, Springer-Verlag, New York-Berlin, 1983.


\end{thebibliography}
\end{document}